\documentclass[journal, 10pt]{IEEEtran}
\usepackage{url}
\usepackage{mathtools, cuted}
\usepackage{stfloats}
\usepackage{hyperref}
\usepackage{amsmath,amsfonts,amssymb}
\usepackage{bbm}
\usepackage{wrapfig}
\usepackage{draftwatermark}
\usepackage{algorithm}
\usepackage{algpseudocode}
\SetWatermarkText{}
\SetWatermarkScale{1}

\newtheorem{assumption}{Assumption}
\newtheorem{lemma}{Lemma}
\newtheorem{theorem}{Theorem}

\newtheorem{example}{Example}
\newtheorem{remark}{Remark}

\usepackage{bm}
\usepackage{mathptmx} 
\usepackage{amsmath}  
\usepackage{amssymb}  
\usepackage{amsfonts}
\usepackage{boxedminipage}
\usepackage{graphicx}
\usepackage[T1]{fontenc}
\usepackage[english]{babel}
\usepackage{pst-sigsys}
\usepackage{pst-node,pst-circ}
\definecolor{dgray}{gray}{0.6}
\definecolor{lgray}{gray}{0.8}
\newcommand{\ignore}[1]{}

\def\bbbz{{\mathchoice {\hbox{$\sf\textstyle Z\kern-0.4em Z$}}
{\hbox{$\sf\textstyle Z\kern-0.4em Z$}}
{\hbox{$\sf\scriptstyle Z\kern-0.3em Z$}}
{\hbox{$\sf\scriptscriptstyle Z\kern-0.2em Z$}}}}

\newcommand{\ba}{\begin{array}}
\newcommand{\ea}{\end{array}}
\newcommand{\beq}{\begin{equation}}
\newcommand{\eeq}{\end{equation}}
\newcommand{\beqy}{\begin{eqnarray}}
\newcommand{\eeqy}{\end{eqnarray}}
\newcommand{\beqyn}{\begin{eqnarray*}}
\newcommand{\eeqyn}{\end{eqnarray*}}
\newcommand{\bi}{\begin{itemize}}
\newcommand{\ei}{\end{itemize}}
\newcommand{\bseq}{\begin{subeqnarray}}
\newcommand{\eseq}{\end{subeqnarray}}

\newcommand{\bex}{\begin{example}}
\newcommand{\eex}{\end{example}}
\newcommand{\bexer}{\begin{exercise}}
\newcommand{\eexer}{\end{exercise}}
\newcommand{\bmct}{\begin{fact}}
\newcommand{\efct}{\end{fact}}

\usepackage{etoolbox}

\patchcmd{\IEEEproofindentspace}{2\parindent}{-2pt}{}{}

\begin{document}

\title{Algorithms for constrained optimal transport}

\author{ Martin Corless\thanks{{\bf Affiliations:} Martin Corless (corless@purdue.edu), is with the School  of Astronautics and Aeronautics, Purdue University, West Lafayette. The other authors are with the Dyson School of Design Engineering, Imperial College London. Anthony Quinn is also with the Dept.\  of Electronic \& Electrical Engineering, Trinity College Dublin. } and Anthony Quinn and Sarah Boufelja Y. and Robert Shorten. 
}

\maketitle

\parindent 0mm

\vspace*{-.4cm}
\begin{abstract}
We derive iterative scaling algorithms of the Sinkhorn-Knopp (SK) type for  constrained optimal transport. The constraints are in the form of prior-imposed zeroes in the  transport plan. Based on classical Bregman arguments, we prove  asymptotic convergence of our algorithms to a unique optimal  solution.  New insights obtained from the convergence proof are highlighted. An example from electrical vehicle charging in a smart city context is outlined, in which the prior zero-constraints prevent energy from being transported from some providers to some vehicles. 
\end{abstract}

\section{Introduction: Optimal Transport (OT)}

Optimal Transport (OT) has emerged as a key enabling mathematical technology which is driving a growing number of contemporary engineering applications in fields such as machine learning, image processing, and in optimization and control \cite{control}, \cite{pmlr-v37-kusnerb15}, \cite{10.1007/978-3-642-15561-1_16}.  The field is an old one, dating back to Monge in the 18th century.
Nevertheless, OT is now attracting accelerated attention from both theoreticians and practitioners alike, with widespread application in engineering practice. OT refers to a class of problems in which we seek to transport a finite  resource (e.g.\ mass) from a distributed source to a distributed target, in some optimal manner. For example, the classical problems of Monge and Kantorovitch relate to the transportation of a distributed pile of sand \cite{vil} from one location to another. The contemporary analogues of these formulations (for example, in economics) arise in resource (re)allocation problems,  and in supply-demand matching problems~\cite{gal}. Closely related problems arise in key machine learning (ML) settings, such as domain adaptation. Here, source data are optimally adapted (i.e.\ transported) for use in a related, but distinct, target domain~\cite{Courty}. An example is ``data repair'',  in which  OT is used to de-bias ML tasks in order to  improve their fairness properties. Classifiers and recommender systems---trained on data with fairness deficiencies such as representation bias, dependence on protected attributes, {\em etc}.---can be fairness-repaired using methods of OT \cite{gor}.  \newline

In this paper, we confine ourselves to the discrete case, in which the resource (mass) is distributed across a finite (though typically large) number of states in both the source and target.  For every pair of source-target states, the designer specifies a cost per unit mass of transporting the resource from this  source state to this target state. The OT problem then involves the design of a transport plan that specifies the amount of resource to be transported for every one of these source-target state pairs, in a way that minimizes the total cost of transport.  A  feature of our setting of the OT problem in this paper is that we impose constraints in which {\em some source states are not allowed to  transport to some target states}. {These prior zero-transport constraints may arise for a variety of reasons. For example, in economics, certain producers may not be allowed to supply a product to certain consumers. In  fairness-aware domain adaptation for ML (above), data for a particular demographic group may not be mapped to a different demographic group.} {These OT problems give rise to  linear programs of very large scale~\cite{pey}. Linear programming solvers are known to scale poorly, and much of the current OT work focuses on regularizing  the classical OT problem to ensure tractability at scale. The best known of these approaches is the {\em entropic regularisation} of the OT problem~\cite{pey}. In essence, this amounts  to the addition of an extra term in the (total cost) objective, which penalizes a specified divergence of the transport plan from a prior-defined ideal plan \cite{BoufeljaAl2023}.  The resulting optimization problem is then strongly convex on a compact convex set, and can be solved efficiently using a Sinkhorn-Knopp-type (SK) iterative scaling algorithm, with its favourable large-scale computational properties~\cite{cut1}, \cite{ger}. These computational OT results and guarantees---based on the SK algorithm--- have been central to the widespread adoption of the OT paradigm in applications, particularly in ML  \cite{vil}, \cite{Courty}, \cite{wasgan},  \cite{article}.}\newline
 
 In contrast to previous results in the literature, our purpose in this paper is to consider OT problems in which some elements of the transport plan are forced to be zero {\em a priori},
 and to derive  iterative scaling algorithms of the SK-type to solve these problems. We use  Bregman-type arguments \cite{Bregman1967} to prove convergence of our algorithms to the (unique) solution of the  OT problems under consideration. We outline an  application of our results in a smart cities context, where prior zeroes in the transport plan are an important constraint.

 \section{Prior work on OT and the SK Algorithm}

 OT problems are often cast in the language of probability, where the resource (mass) to be transported is a probability measure, and so the net resource (mass) is  normalized to one. The transport plan is thus a bivariate  probability mass function (pmf), and the source and target distributions are its two marginals. We speak of a transport plan that {\em moves} the source distribution (marginal) to the target distribution (marginal).  In this way,  OT can be viewed as a mathematical framework for measuring the distance between (source and target marginal) probability distributions, this distance being  the minimal expected cost associated with the optimal transport plan~\cite{vil}. This Kantorovich-Rubinstein (also called Wasserstein or earth-mover's) norm~\cite{vil} for classical entropy-regularized OT gives rise to a unique solution of a strongly convex optimization problem on a compact convex support~\cite{brualdi_1968}, \cite{BaradatAl2022}, with linear convergence of the SK iterative scaling algorithm to this unique solution~\cite{doi:10.1137/060659624}. If, {\em as in this paper}, we impose prior zero-constraints on some elements of the transport plan, 
 there may not exist a plan that satisfies exact marginal constraints at the source and target, as already noted. Another situation in which a  transport plan satisfying exact marginal constraints  does not exist is unbalanced OT (UOT) \cite{BaradatAl2022}, in which the prescribed net masses of the  source and target are not the same.
 Earlier, in~\cite{ChizatAl2018}, an SK-type iterative scaling algorithm was derived for unconstrained UOT  (i.e.\ without prior zero-constraints on the transport plan), by relaxing the   source and target marginal constraints via  extra terms in the cost that penalize noncompliance with the marginal constraints. This was shown to converge linearly to a unique solution. 
 In \cite{BaradatAl2022}, the authors study the UOT 
 case with prior zero-constraints on the transport plan.
 They prove that the SK algorithm has two convergent subsequences in this case and they propose a transport plan which is the component-wise geometric mean of 
 the limits plans of the two convergent subsequences
  \newline
 
The paper is organized as follows: in Section~\ref{sec:OTzero}, our OT problem is specified with prior constraints on the pattern of the transport plan. After developing background material, two  iterative scaling algorithms are presented  in Section \ref{sec:algore's_favorite_dance} for OT problems with prior zero-constraints.  
Section V presents some properties of the resulting OT plans.
Section  VI provides the proof of convergence or our algorithms.
Finally, a use-case from the domain of smart mobility
is outlined in Section \ref{sec:experiment}, illustrating  the efficacy of our new algorithms.   \newline

\section{The Optimal transport (OT) problem  with prior zero-constraints}
\label{sec:OTzero}
Consider an OT problem in which we  have a finite number, $m$, of {\em source agents}  (i.e.\ source states), $i=1,2, \dots, m$,  and a finite number, $n$, of {\em target agents} (i.e.\ target states),  $j= 1,2, \dots, n$.
Suppose each source agent, $i$, has capacity or mass,  $\tilde{u}_i  >0$,  and each target agent, $j$, desires a capacity or mass, $\tilde{v} _j >0$.
We wish to transport the masses of the source agents to the target agents.
For the time being,  we impose mass balance (conservation), i.e.\ 
\begin{align}
\label{eq:ut=vt}
\sum_{i=1}^m \tilde{u}_i =\sum_{j=1}^n \tilde{v}_j 
\end{align}
that is, ${\bm 1}'\tilde{u} = {\bm 1}'\tilde{v}$  where   ${\bm 1}$ is a vector of ones, of appropriate length. Let
$
t_{ij} \ge 0
$
be the mass or capacity that source agent, $i$,  transports to  target agent, $j$.
We will refer to  the $m \times n$ matrix, $T=\left\{t_{ij}\right\}$, as the {\em transport plan}.
As a consequence of the capacity constraints of the source and target agents, we must have
\begin{align}
\label{eq:rowconstraints}
\sum_{j=1}^nt_{ij}  &= \tilde{u}_i,  \qquad i=1, \dots, m
\end{align}
or, $T {\bm 1} = \tilde{u}$, and
\begin{align}
\label{eq:colconstraints}
\sum_{i=1}^m t_{ij} = \tilde{v}_j,  \qquad j=1, \dots, n
\end{align}
or, $T' {\bm 1} = \tilde{v}$.
Note that constraints  \eqref{eq:ut=vt}, \eqref{eq:rowconstraints}  and \eqref{eq:colconstraints} imply that 
\begin{equation}
\sum_{i=1}^m\sum_{j=1}^n t_{ij} =\sum_{i=1}^m \tilde{u}_i = \sum_{j=1}^n \tilde{v}_j\
\label{eq:bal}
\end{equation}

Next, suppose that a  specific source agent, $i$, never  transports to a specific target agent, $j$, so that 
 $t_{ij}= 0$.
We let $\mathcal {Z}$ be the set of these prior non-transporting  index pairs:
\begin{equation}
\label{eq:U}
t_{ij}= 0  \text{ if } (i ,j) \in \mathcal {Z} 
\end{equation}
The set of allowable transport plans is given by
\begin{align}
\mathcal{T} \equiv \left\{ T \in \mathbb{R}^{m\times n} :  \begin{array}{ll} t_{ij}  \ge 0 &\forall  (i,j) \notin  \mathcal{Z} \\ 
t_{ij} = 0 &\forall (i,j) \in  \mathcal{Z} \end{array}  \right\}
\label{eq:Tset}
\end{align}
We assume that for every $i$, that there is at least one $(i, j) \notin \mathcal{Z}$
and for every $j$, there is also at least one $(i, j) \notin \mathcal{Z}$;
i.e.\ if 
$t_{ij} >0$ for all $(i, j) \notin \mathcal{Z}$
then, $T$ has   no zero-rows or zero-columns.
If $\mathcal{Z}$ is empty, i.e.\ if every source agent can transport to every target agent, then the marginal constraints \eqref{eq:rowconstraints} and \eqref{eq:colconstraints} are feasible (i.e. there is at least one solution in $\mathcal{T}$, for example, 
$T \equiv  \tilde{u}\tilde{v}'/{\bm 1}'\tilde{v}$).
 However, if $\mathcal{Z}$ is non-empty, then \eqref{eq:rowconstraints} and \eqref{eq:colconstraints}  may not be feasible (i.e.\ there may not be any solution in $\mathcal{T}$). An example is given  by  $m=n=2$ and  $\mathcal{Z} \equiv \{(2,2)\}$: 
\[
T=\left(\begin{array}{cc}
t_{11} & t_{12}\\
t_{21}& 0
\end{array}
\right)
\qquad \mbox{with} \qquad  \tilde{u} = \tilde{v} =\left(\begin{array}{c} 1\\ a
\end{array}
\right), \;\;  a>1
\]

\paragraph{Transport cost.} Suppose that  there is a cost, $c_{ij} \in \mathbb{R}$, of moving a unit mass from source, $i$, to
target,  $j$.
Then, the total cost of transport is
$
\sum_{ (i,j) \notin  \mathcal{Z} } c_{ij}t_{ij} 
$
This gives rise to the {\em constrained OT problem:}
\begin{eqnarray}
&&\min _{T \in \mathcal{U}(\tilde{u}, \tilde{v})} \sum_{(i,j)\notin  \mathcal{Z} } c_{ij}t_{ij},  \label{eq:OptOrig}\\
\mathcal{U}(\tilde{u},\tilde{v}) &= &\left\{ T \in \mathcal{T} : \small{\text{constraints \eqref{eq:rowconstraints} and \eqref{eq:colconstraints}  hold}}\right\}  \label{eq:ScriptT}
\end{eqnarray}

\paragraph{An ideal transport plan and regularization}
Next, suppose that it is also desirable for $T$ to be close to some {\em ideal} (desired) transport plan, $\tilde{T} \in \mathcal{T}$ (\ref{eq:Tset}) with 
$\tilde{t}_{ij} >0$ for $(i,j) \notin \mathcal{Z}$.
We therefore  modify 
the cost to
\begin{equation}
\label{eq:newJ}
\sum_{(i,j) \notin \mathcal{Z}} c _{ij} t_{ij} +\gamma_0 KL(T|\tilde{T})
\end{equation}
where  $\gamma_0 >0$ is a pre-assigned regularization constant, and the {\em Kullback-Leibler ($KL$) divergence}  \cite{BoufeljaAl2023} of $T$ from $\tilde{T}$ is defined as
\begin{align}
KL(T| \tilde{T}) \equiv   \sum_{(i,j) \notin \mathcal{Z}} kl(t_{ij}|\tilde{t}_{ij})	
\label{eq:KL}
\end{align}
where,  for any two scalar, $t\ge 0$ and $\tilde{t}>0$, we define
\[
kl(t|\tilde{t}) \equiv
\left\{\begin{array}{ccc}
\tilde{t}	&\text{if}	& t = 0	\\
\displaystyle{ t \log \left(\frac{t}{\tilde{t}}\right)} -t + \tilde{t}	& \text{if}	& t >0
\end{array}
\right.
\]
One may readily show that,  for all $t\ge0$ and $\tilde{t}>0$,  $kl(t|\tilde{t}) \ge 0$; also $kl(t|\tilde{t}) =0$ iff  $t=\tilde{t}$. Hence, for all $T\in \mathcal{T}$,  $KL(T|\tilde{T}) \ge 0$ and $KL(T|\tilde{T}) =0$ iff $T=\tilde{T}$.\newline

\begin{remark}{{Note again that original optimization problem  \eqref{eq:OptOrig} is a linear programming problem and may be computationally burdensome to solve. So, a common solution  approach 
is approximately to solve the  original  problem  \eqref{eq:OptOrig} by minimizing  the regularized cost  in   \eqref{eq:newJ} with $\gamma_0$ small but postive.
This problem can be solved more efficiently than the original problem.}}\end{remark}

Next, note that  
$
c_{ij}t_{ij}  	
=\gamma_0  t_{ij}\log\left(\frac{1}{\exp(-c _{ij}/\gamma_0)}\right)
$
and define 
{\small
\begin{align}
k_{ij} & \equiv \left\{
\begin{array}{ccl}
\tilde{t}_{ij}\exp(-c_{ij}/\gamma_0) \qquad & \text{if} \qquad & (i,j) \notin \mathcal{Z}	\\
0	\qquad & \text{if} \qquad & (i,j)  \in \mathcal{Z}
\end{array}
\right.	
\label{eq:K}
\end{align}
}
Then, whenever $t_{ij} \neq 0$,
\begin{align}
& c _{ij} t_{ij} +\gamma_0 kl(t_{ij}|\tilde{t}_{ij}) \nonumber \\
& 
= \gamma_0 t_{ij} \log\left( \frac{\tilde{t}_{ij}}{k_{ij}} \right)
+\gamma_0 t_{ij} \log\left(\frac{ t_{ij}}{\tilde{t}_{ij}} \right) - \gamma_0 t_{ij}
 +\gamma_0  \tilde{t}_{ij}	
\nonumber  \\
  &= \gamma_0 t_{ij} \log\left( \frac{t_{ij}}{k_{ij}} \right)
 - \gamma_0 t_{ij} +\gamma_0  k_{ij}
 +\gamma_0  \tilde{t}_{ij}
  -\gamma_0  k_{ij}
\nonumber  \\ 
  &= \gamma_0 kl(t_{ij}| k_{ij}) +\gamma_0  \tilde{t}_{ij}
  -\gamma_0  k_{ij}
  \label{eq:newkl}
\end{align}
\eqref{eq:newkl} also holds for  $t_{ij} =0$ , $(i,j) \notin \mathcal{Z}$.
Hence, cost \eqref{eq:newJ} can be expressed as 
$
\gamma_0 KL(T|K) +\sum_{(i,j) \notin \mathcal{Z}} \gamma_0 (\tilde{t}_{ij} - k_{ij})
$.
Therefore,  we consider the new optimization problem:
\begin{equation}
\label{eq:OptReg}
\min_{T \in {\mathcal U}(\tilde{u},\tilde{v}) }
KL(T|K)
\end{equation}
where $\mathcal{U}(\tilde{u},\tilde{v})$ is given by \eqref{eq:ScriptT}. 
When the optimization problem \eqref{eq:OptReg} is feasible (i.e.\ $\mathcal{U}(\tilde{u},\tilde{v})$ is non-empty),
its solution can be obtained using the Sinkhorn-Knopp (SK) Algorithm \cite{doi:10.1137/060659624}.
\newline

{\em Sinkhorn-Knopp (SK) Algorithm.}
Initialize
$
  d_{2j}(0) = 1
$. Iterate for $l= 0,1, \dots$
\begin{align}
d_{1i}(l+1)  &=  \frac{\tilde{u}_i}{\Sigma_{j=1}^n k_{ij}d_{2j}(l)}	
	\\
d_{2j}(l+1) &= \frac{\tilde{v}_j}{\Sigma_{i=1}^m d_{1i}(l+1)k_{ij}}	\\
t_{ij}(l+1) &= d_{1i}(l+1)k_{ij}d_{2j}(l+1) 
\end{align}
$\protect{\hspace*{7.75cm}} \square $


If \eqref{eq:OptReg} is feasible,   the sequence, $\{T(l)\}$ converges to a limit, $T^*$, which 
is the unique minimizer for \eqref{eq:OptReg} \cite{BaradatAl2022}.
 If  $t^*_{ij} = 0$ for some $(i,j) \notin \mathcal{Z}$,  then either
$\{d_{ii}(l)\}$ or $\{d_{2j}(l)\}$ do not converge.
If \eqref{eq:OptReg} is not feasible,
   the sequence, $\{T(l)\}$ has two convergent subsequences with different limits \cite{BaradatAl2022}.
\newline

\paragraph{When desired marginals constraints are not achievable}
If $\mathcal{Z}$ is non-empty,   it may not be possible to simultaneously  satisfy the marginal constraints,  \eqref{eq:rowconstraints} and \eqref{eq:colconstraints}.
If---as already imposed in (\ref{eq:Tset})---$T$ has no enforced rows of zeroes,   the row constraints can always be achieved, as explained in Remark \ref{rem:feasible}, below. Therefore, consider the situation in which row constraints \eqref{eq:rowconstraints} are satisfied but 
column constraints \eqref{eq:colconstraints}  are not necessarily satisfied.
Instead of requiring satisfaction of the column constraints, we add a KL-based penalty term to the cost function, which is a measure of non-compliance with the column constraints, $\tilde{v}$.
We therefore consider a new optimization problem given by
\begin{equation}
   \min_{ T \in \mathcal{V}(\tilde{u})} KL(T|K)
+\gamma KL(v_T|\tilde{v})
\label{eq:GenOpt}
\end{equation} 
where $KL(v_T|\tilde{v}) =\sum_{j=1}^n kl(v_{T_j}| \tilde{v}_j)$, 
$v_{T_j}  {\equiv} \sum_{i=1}^m  t_{ij}$, $\gamma >0$,
\begin{align}
 \mathcal{V}(\tilde{u}) &\equiv \left\{ T\in \mathcal{T} : \sum_{j=1}^n t_{ij} = \tilde{u}_i, \; i=1, \dots, m \right\}.
 \label{eq:Vu}
\end{align}

\begin{remark}
\label{rem:feasible}
For any admissible $\tilde{u}$,  one can always obtain a $T$ in $\mathcal{V}(\tilde{u})$ with $t_{ij} > 0$ for all $(i,j) \notin \mathcal{Z}$. Simply choose any $T\in \mathcal{T}$ 
with  $t_{ij} > 0$ for all $(i,j) \notin \mathcal{Z}$, and scale each of its rows to yield a new matrix in $\mathcal{V}(\tilde{u})$.
Therefore, $\mathcal{V}(\tilde{u})$ is non-empty, and since we are minimizing a continuous strictly convex function (\ref{eq:GenOpt}) over a compact convex set (\ref{eq:Vu}),
a unique minimizer exists.
\end{remark}
$\protect{\hspace*{7.75cm}} \square $

\begin{remark}{{\cite{RinghAl2022} considers a generalization of the problem considered here. To highlight the novelty of our work we make the following comments.}}\newline
\begin{itemize}
\item[(a)] 
{It is true that our problem  is a special case of the {\bf initial}  general problem considered in  \cite{RinghAl2022} by, among other things,  specializing to the bivariate (so, bi-marginal) case, and taking $c_{ij} =
\infty$ for all entries in the transport plan that are to  be zero.}\newline

\item[(b)] {
However,
our algorithms---which have been designed, among other things,  to be relevant  in important applications of OT, such as the sharing economy application of our Section~VII---are {\em not\/}  presented in \cite{RinghAl2022}. }\newline

\item[(c)] {More substantially, we note the following difference. In order to prove convergence of their  algorithm to solve their general problem, it is assumed that all the elements of their cost {tensor}  $\mathbf C$ are finite (Assumption C {on page 5, left column,  of \cite{RinghAl2022}}). Under this assumption, our problem is not a special case of the  problem solved by the algorithm in \cite{RinghAl2022}, because for our problem to be a special case, some of the elements of their cost tensor $\mathbf C$ must be infinite. This assumption is used in the proof of Lemma III.8  in \cite{RinghAl2022}. {Note also  that prior zeros are {\em not\/} imposed on their transport tensor, $M$ (corresponding to our transport matrix, $T$).} \newline}
\end{itemize}
\end{remark}

\section{Main result: a modified SK  algorithm}
\label{sec:algore's_favorite_dance}
{We now present the main results of the paper. These results yield iterative scaling algorithms that produce a sequence, $\{T(l\}$, converging to the optimal transport plan, $T^*$, for 
\eqref{eq:GenOpt}. 
In particular, two algorithms are presented. Algorithm 1 is obtained using  results in  \cite{Bregman1967}. Algorithm 2 is equivalent to Algorithm I and is presented for comparison to other related algorithms such as the SK Algorithm. 
A third algorithm, the Chizat Algorithm  is a related algorithm from the literature that is included to provide context for our contributions.}
\newline

{\bf Algorithm 1.}
Initialize
$
T(0)= K, v(0) = \tilde{v}
$.
Iterate for $l=0,1, \dots$
\begin{align}
c_{1i}(l+1) 	&=  \frac{\tilde{u}_i}{\Sigma_{j=1}^n t_{ij}(l)} 	\label{eq:c_1Iterate}	\\
c_{2j}(l+1)  &= \left(\frac{v_j(l)}{\Sigma_{i=1}^m c_{1i}(l+1)t_{ij}(l)}\right)^{\frac{\gamma }{1+\gamma}}	
\label{eq:c_2Iterate}\\
  t_{ij}(l+1)   &= c_{1i}(l+1)t_{ij}(l)c_{2j}(l+1)	
 \label{eq:t_ijIterate}  \\
  v_{j}(l+1) &= c_{2j}(l+1)^{-1/\gamma} v_j(l)	
  \label{eq:v_jIterate} 
 \end{align}
$\protect{\hspace*{7.75cm}} \square $

The following theorem provides the main result of this paper. A proof is provided in Section \ref{sec:Proof}.\newline

\begin{theorem}
\label{th:MSKA}
Consider the  sequence, $\{T(l),v(l)\}$, generated by  
Algorithm 1.
This sequence converges to the (unique) limit, $(T^*, v^*)$, 
and $T^*$ is the minimizer for  the optimization problem   given by \eqref{eq:GenOpt}.\newline
\end{theorem}

\begin{remark}
Since 
\begin{align}
\label{eq:row*constraint}
    \Sigma_{j=1}^n t^*_{ij} =\tilde{u}_i
\end{align}
 from  \eqref{eq:c_1Iterate},
$
\lim_{l\rightarrow \infty}c_{1i}(l) = 1
$
  for all $i$.
  In Lemma \ref{lem:lem1b} (see Section~\ref{sec:propmin}), it will be shown that $t^*_{ij} >0$ for all $(i,j) \notin \mathcal{Z}$.
Hence, $T^*$ is guaranteed  not to contain any row of zeroes, and it follows from \eqref{eq:t_ijIterate} that
$
     \lim_{l\rightarrow \infty}c_{2j}(l) = 1
$
  for all $j$.
  From \eqref{eq:c_2Iterate}, we see that
  \begin{align}
  \label{eq:col*constraint}
   \Sigma_{i=1}^m t^*_{ij}  = v^*_j
  \end{align}
  and, since $T^*$ does not contain any column of zeroes, we must have
   $   v^*_j >0$,
   $\forall j$.
It now follows, from \eqref{eq:row*constraint} and \eqref{eq:col*constraint},
that
$
\sum_{j=1}^n v_j^* = \sum_{i=1}^m \tilde{u}_i
$
and so $v^*$ and $\tilde{u}$ are {\em balanced } (i.e. mass-conserving).  
The above result holds even in the {\em unbalanced} (i.e.\ non-mass-conserving) problem,
$
\sum_{i=1}^m \tilde{u}_i \neq \sum_{j=1}^n \tilde{v}_j
\label{eq:nonmc}
$.
{In this case, application of the SK algorithm
produces a
 sequence, $\{T(l)\}$, which  has two convergent subsequences with different limits \cite{BaradatAl2022}.}
\end{remark}
$\protect{\hspace*{7.75cm}} \square $

{In order to obtain an algorithm for comparison with the SK algorithm,} we  introduce scaling parameters, $d_{1i}(l)$ and $d_{2j}(l)$, defined by
\begin{align}
d_{1i}(l+1) &\equiv c_{1i}(l+1)d_{1i}(l), \qquad d_{1i}(1) \equiv c_{i1}(1)	\\
d_{2j}(l+1) &\equiv c_{2j}(l+1)d_{2j}(l), \qquad d_{2j}(0) \equiv 1
\end{align}
Then
\begin{align}
t_{ij}(l+1)   &= d_{1i}(l+1)k_{ij}d_{2j}(l+1)	
 \label{eq:t_ijIterate2}  \\
  v_{j}(l+1) &= d_{2j}(l+1)^{-1/\gamma} \tilde{v}_j	
  \label{eq:v_jIterate2} 
\end{align}
The  sequence, $\{T(l)\}$, obtained from Algorithm~1,  can then also be obtained from the following algorithm.
\newline

{\bf Algorithm 2.}
Initialize
$
  d_{2j}(0) = 1.
$
Iterate for $l= 0,1, \dots$
\begin{align}
d_{1i}(l+1)  &=  \frac{\tilde{u}_i}{\Sigma_{j=1}^n k_{ij}d_{2j}(l)}			\label{eq:d1Iterate}
	\\
d_{2j}(l+1) &= \left(\frac{\tilde{v}_j}{\Sigma_{i=1}^m d_{1i}(l+1)k_{ij}}\right)^\frac{\gamma}{1+\gamma}	\label{eq:d2Iterate}\\
t_{ij}(l+1) &= d_{1i}(l+1)k_{ij}d_{2j}(l+1)
\label{eq:t_ijIterate3}
\end{align}
$\protect{\hspace*{7.75cm}} \square $

\begin{remark}
\label{rem:d}
Since $\tilde{v}_j, v^*_j> 0$ for all $j$, it follows from \eqref{eq:v_jIterate2}
that the sequence $\{d_{2j}(l)\}$ has a limit $d^*_{2j}$ which is non-zero
for all $j$. Now
\eqref{eq:t_ijIterate2} implies that
that the sequence $\{d_{1i}(l)\}$ has a limit $d^*_{1i}$
for all $i$.
Moreover \eqref{eq:d1Iterate} and \eqref{eq:d2Iterate} show that these limits satisfy
\begin{equation}
 d^*_{1i} =\frac{ \tilde{u}_i} {\sum_{j=1}^n  k_{ij} d^*_{2j}}
\qquad  d^*_{2j} = \left(\frac{ \tilde{v}_j}{\sum_{i=1}^m  d^*_{1i} k_{ij}}\right)^{\frac{\gamma}{1+\gamma}}
\label{eq:dLimits}
\end{equation}
and \eqref{eq:t_ijIterate3} results in
\begin{equation}
\label{eq:tLimit}
t^*_{ij} = d^*_{1i}k_{ij}d^*_{2j}
\end{equation}

\end{remark}
$\protect{\hspace*{7.75cm}} \square $

\begin{remark}
If one considers the limit, as $\gamma \rightarrow \infty$, in Algorithm~2,  one obtains the SK algorithm \cite{doi:10.1137/060659624}, whose  sequence, $\{T(l)\}$,
converges to the minimizer for optimization problem \eqref{eq:OptReg}, 
{\em provided this problem is feasible.}
\end{remark}
$\protect{\hspace*{7.75cm}} \square $

\begin{remark}
{Before proceeding we note that \cite{ChizatAl2018} considers a problem which is related (but different from) to the problem considered here. Namely, to solve}
{
\begin{equation}
\label{eq:ChizatProb}
\min_{ T \in \mathcal{T}} KL(T|K)
+\gamma_1 KL(u_T|\tilde{u}) +\gamma_2 KL(v_T|\tilde{v})  
\end{equation}
where
$
u_{T_i} = \sum_{j=1}^n  t_{ij}$ for all $i$,
 $
k_{ij}  =
\exp(-c_{ij}/\gamma_0) \text{ for all }  (i,j)
$
and  $\gamma_1, \gamma_2 >0$.
In \cite{ChizatAl2018}, neither  marginal constraint is enforced and  
none of the elements of $T$ are constrained to be zero.
This is in contrast to the current paper, in which the  source  marginal is constrained to be $\tilde{u}$, and some of the elements of $T$ 
are enforced to be zero, specifically
$t_{ij} = 0$ for $(i,j) \in \mathcal{Z}$. 
\cite{ChizatAl2018} shows that the sequence, $\{T(l)\}$, generated by the following algorithm converges to a limit which is the minimizer for optimization
problem \eqref{eq:ChizatProb}.}\newline
\end{remark}

{{\em {\bf Chizat Algorithm.}}}
Initialize
$
  d_{2j}(0) = 1.
$
Iterate for $l= 0,1, \dots$
\begin{align}
d_{1i}(l+1)  &=  \left(\frac{\tilde{u}_i}{\Sigma_{j=1}^n k_{ij}d_{2j}(l)}\right)^\frac{\gamma_1}{1+\gamma_1}
	\\
d_{2j}(l+1) &= \left(\frac{\tilde{v}_j}{\Sigma_{i=1}^m d_{1i}(l+1)k_{ij}}\right)^\frac{\gamma_2}{1+\gamma_2}\\
t_{ij}(l+1) &= d_{1i}(l+1)k_{ij}d_{2j}(l+1)
\end{align}
$\protect{\hspace*{7.75cm}} \square $

\begin{remark}
{Remarkably, if  one considers the limit as $\gamma_1 \rightarrow \infty$, 
 then the Chizat algorithm reduces to our Algorithm 2 in the case where there are no constraints on $T$ and all the elements
of $\tilde{T}$ equal one.}
\end{remark}

\section{Properties  of the minimizer, $T^*$ }
\label{sec:propmin}
\begin{lemma}
\label{lem:lem1}
If 
$T^*$ is a minimizer for   \eqref{eq:GenOpt}, then $t^*_{ij} >0 $ for all $(i,j)  \notin \mathcal{Z}$.
\end{lemma}

\begin{proof}
Suppose that $T^*$ is a minimizer for   \eqref{eq:GenOpt}.
From Remark \ref{rem:feasible},  there is a $\hat{T} \in \mathcal{V}(\tilde{u})$
with $\hat{t}_{ij} >0$ for all $(i,j) \notin \mathcal{Z}$.
Since  $\mathcal{V}(\tilde{u})$ is convex,  $(1-\lambda)T^* + \lambda \hat{T}$ is in  $\mathcal{V}(\tilde{u})$ for all  $\lambda \in [0, 1]$.
Also,
 there are bounds,   $\beta_1$ and $\beta_2$, such that, if  $t^*_{ij} >0$ then  for any $(i,j) \notin \mathcal{Z}$,
$
0 < \beta_1 \le t_{ij} \le \beta_2 
$
for $T =(1-\lambda)T^* + \lambda \hat{T} \text{  and } \lambda \in [0, 1]$.
The function to be minimized in \eqref{eq:GenOpt}  can be expressed as
\[
f(T) =
\sum_{(i,j) \notin {\mathcal Z}}   kl(t_{ij} | k_{ij}) 
 +\gamma \sum_{j=1}^n kl \left(\sum_{l=1}^m t_{lj} \big| \tilde{v}_j \right)	
\]

Consider any  $\lambda \in (0, 1]$ and $T=(1-\lambda)T^* + \lambda \hat{T}$.
By the mean value theorem, $\exists$  $\underline{\lambda} \in (0, \lambda)$ such that
\begin{align}
\label{eq:f(T)}
f(T)
= f(T^*) +  \lambda\!\sum_{(i,j) \notin {\mathcal Z}}   \frac{\partial f}{\partial t_{ij}}(\underline{T}) (\hat{t}_{ij} - t^*_{ij})
\end{align}
where $\underline{T} = (1-\underline{\lambda})T^* + \underline{\lambda} \hat{T}$ and
\beq
 \frac{\partial f}{\partial t_{ij}}(\underline{T}) = \log \left( \frac{\underline{t}_{ij}}{k_{ij} }  \right) 
 + \gamma   \log \left(\frac{\sum_{l=1}^m \underline{t}_{lj}}{\tilde{v}_j}\right)
 \eeq
 
If  $t^*_{ij} >0$ then, for all $\lambda \in (0, 1]$,  we have  $0 < \beta_1 \le \underline{t}_{ij}\le \beta_2$,
and $0 < \beta_1 \le \sum_{l=1}^m \underline{t}_{lj} \le \beta_3$ for some $\beta_3$.
Hence
$\frac{\partial f}{\partial t_{ij}}(\underline{T})(\hat{t}_{ij}- t^*_{ij}) \le \gamma_{ij}
$
  for some $\gamma_{ij}$.
Suppose that  $t^*_{ij} =0$   for some $(i,j) \notin \mathcal{Z}$. Then $\lim_{\lambda \rightarrow 0} \underline{t}_{ij} = t^*_{ij} = 0$; and
$
\lim_{\lambda \rightarrow 0} \frac{\partial f}{\partial t_{ij}}(\underline{T})(\hat{t}_{ij}- t^*_{ij}) = -\infty
$.
This implies  that, for $\lambda >0$ sufficiently small,
$
\sum_{(i,j) \notin \mathcal{Z}}  \frac{\partial f}{\partial t_{ij}}(\underline{T}) (\hat{t}_{ij} - t^*_{ij}) <0
$
which along with \eqref{eq:f(T)} yields the contradiction,
$
f(T)  < f(T^*).
$

\end{proof}

The following result can be obtained from the discussion of Algorithm 2; see Remark \ref{rem:d}.
However we wish to provide a proof which is independent of any algorithm.\newline

\begin{lemma}
\label{lem:lem1b}
A matrix, $T^*$,   solves the OT problem given by \eqref{eq:GenOpt}  iff
$\exists$   positive scalars, $d_{11}, \dots,  d_{1m}$ and $d_{21},\dots, d_{2n}$,  such that, 
for all $(i, j)$
\ignore{
\begin{equation}
t^*_{ij} = d_{1i}k_{ij} d_{2j}
\label{eq:TKscale2}
\end{equation}
and
\begin{equation}
 d_{1i} =\frac{ \tilde{u}_i} {\sum_{j=1}^n  k_{ij} d_{2j}}
\qquad  d_{2j} = \left(\frac{ \tilde{v}_j}{\sum_{i=1}^m  d_{1i} k_{ij}}\right)^{\frac{\gamma}{1+\gamma}}
\label{eq:Constraints}
\end{equation}
}
\eqref{eq:tLimit} and \eqref{eq:dLimits} hold.
\end{lemma}

\vspace*{.5cm}

 \begin{proof}  
If
   $(i,j) \in \mathcal{Z}$, we have  $t^*_{ij} = k_{ij} =0$  and  the expression for $t^*_{ij}$ in \eqref{eq:tLimit} holds. 
The Lagrangian associated with this optimization  problem is
\begin{align*}
L(T, \alpha) & =    \sum_{(i,j) \notin {\mathcal Z}}   kl(t_{ij} | k_{ij}) 
 +\gamma \sum_{j=1}^n kl \left(\sum_{l=1}^m t_{lj} | \tilde{v}_j \right)	\\
&+\sum_{i=1}^m \alpha_i \left(\sum_{l=1}^n t_{il}- \tilde{u}_i \right)
\end{align*}

For $(i,j) \notin \mathcal{Z}$ and $t_{ij} \neq 0$, 
\begin{align*}
\frac{\partial L}{\partial t_{ij}}(T, \alpha) &  =   \log \left( \frac{t_{ij}}{k_{ij} }  \right) + \gamma   \log \left(\frac{\sum_{l=1}^m t_{lj}}{\tilde{v}_j}\right) 
+\alpha_i  
\end{align*}
From Lemma \ref{lem:lem1}, we know that $t^*_{ij} >0$ for all $(i,j) \notin \mathcal{Z}$.
Hence $T^*$ is a minimizer iff $\exists$ scalars, $\alpha_1, \dots, \alpha_m$, such that
$\frac{\partial L}{\partial t_{ij}}(T^*,\alpha) = 0$ for all $(i,j) \notin \mathcal{Z}$, that is,
$\log \left( \frac{t^*_{ij}}{k_{ij} }  \right) =  -\alpha_i -\beta_j$,	
where
\begin{align}
\label{eq:beta}
 \beta_j \equiv \gamma   \log \left(\frac{\sum_{l=1}^m t^*_{lj}}{\tilde{v}_j}\right)  
\end{align}
Hence
\begin{align}
t^*_{ij}&= k_{ij}\exp \left( -\alpha_i -\beta_j \right)	
 = \exp(  -\alpha_i )	
k_{ij}  \exp  (- \beta_j)	\nonumber \\					
 &= d_{1i}k_{ij}d_{2j} 
 \label{eq:tij2}	
  \end{align}
where
$
 d_{1i}  \equiv \exp(   -\alpha_i )$ and 
$d_{2j}  \equiv\exp  (- \beta_j)$.
Using the row constraints, we have
$\tilde{u}_i = \sum_{j=1}^m t^*_{ij} =\sum_{j=1}^m d_{1i}k_{ij}d_{2j}$.
Hence,
$d_{1i} = \frac{\tilde{u}_i}{\sum_{j=1}^mk_{ij}d_{2j}}$.
It follows, from  \eqref{eq:beta} and \eqref{eq:tij2}, that
\begin{align*}
d_{2j} \sum_{i=1}^m d_{1i}k_{ij}  = \sum_{i=1}^m t^*_{ij} =  \tilde{v}_j e^{\beta_j/\gamma} = \tilde{v}_j d_{2j}^{-1/\gamma}
\end{align*}
that is,
$
d_{2j}^{ \frac{1+\gamma}{\gamma}}\sum_{i=1}^md_{1i}k_{ij}= \tilde{v}_j$,
or
\[
d_{2j}  =\left( \frac{ \tilde{v}_j }{ \sum_{i=1}^m d_{1i}k_{ij} } \right)^{\frac{\gamma}{ 1+ \gamma}} 
\]

\end{proof}

\section{Convergence of the Modified SK algorithms}
\label{sec:Proof}
To prove  Theorem \ref{th:MSKA}, we need a result from \cite{Bregman1967}.
Let
\[
\mathcal{X} = \left\{ x \in \mathbb{R}^q: x_i >0 \text{ for } i =1,2, \dots, q \right\}
\]
and, for a fixed $\tilde{x} \in \mathcal{X}$,  consider  the strictly convex function, $f: \bar{\mathcal X}\rightarrow \mathbb{R}$, given by 
\begin{equation}
f(x) =KL(x|\tilde{x})
= \sum_{i=1}^q kl(x_i|\tilde{x}_i)
\end{equation}
where $\overline{\mathcal X}$ is the closure of $\mathcal{X}$,
that is 
\[
\overline{\mathcal X} =  \mathbb{R}^q_+ =\left\{ x \in \mathbb{R}^q : x_i \ge 0 \text{ for } i =1,2, \dots, q\right\} 
\]
Also, let
$A_i \in \mathbb{R}^{m_i\times q}$ and $b_i\in \mathbb{R}^{m_i}$,  for $i=1,2,\dots, N$, for some positive integers, $N$ and $m_i$, and consider the following optimization problem: 
\begin{align}
\label{eq:opt0}
\min_{x \in \overline{\mathcal X} }f(x) \qquad s.t.\quad  A_i x = b_i \quad i=1, \dots, N
\end{align}

Let  $\mathcal{C}_i$ be the closed convex set, $\{ x \in \mathbb{R}^q : A_ix=b_i\}$, and assume that 
$\mathcal{C} \equiv \bigcap_{i =1}^N\mathcal{C}_i$ is non-empty.
We also require  the following assumption.\newline

\begin{assumption}
\label{ass1}
For each 
$i=1,2,\dots, N$ and $x\in\mathcal{X}$, $\exists$ $y^*\in \mathcal{X}\cap \mathcal{C}_i$ such that
\begin{align}
\label{eq:opti}
KL(y^* |x) = \min_{y\in \mathcal{X}\cap\mathcal{C}_i} KL(y |x)
\end{align}
\vspace{.5em}
\end{assumption}
Note that, in this assumption, optimization is over \mbox{$\mathcal{X}\cap \mathcal{C}_i$} and not over $\overline{\mathcal X}\cap \mathcal{C}_i$.
We will denote the point $y^*$ above by $P_i(x)$ and refer to it  as the {\em $KL$-projection of $x$ onto $\mathcal{C}_i$.} Let $p$
be a permutation on $\{1,2,\dots, N\}$, that is
\[
p^k(i) \neq i \text{ for } k=1,\dots, N-1 \qquad  p^N(i) = i
\]
where $p^k$ is the application of $p$, $k$ times.\newline

The  following result may be gleaned from  \cite{Bregman1967}.\newline

\begin{theorem}
\label{th:Bregman}
Suppose  $\mathcal{C}\cap \overline{\mathcal X}$ is non-empty, that Assumption \ref{ass1} holds, and
\begin{align}
x(l+1) &= P_{i_{l}}(x(l))		\quad \text{with} \quad x(0) \in {\mathcal X}
\end{align}
where $i_{l+1} = p(i_l)$.
Then,
$\lim_{l \rightarrow \infty }x(l) = x^* \in {\mathcal C}\cap \overline{\mathcal X}$.
Moreover, if
\begin{equation}
\label{eq:range}
\nabla f(x(0)) \text{ is  in the range of }
\left( A'_1\; A'_2\,  \cdots\, A'_m
\right)
\end{equation}
where $'$ denotes transpose,
then
\begin{align}
f(x^*) = \min_{x \in \overline{\mathcal X} }f(x) \quad s.t.\quad  A_ix = b_i \quad i=1, \dots, N
\end{align}
\end{theorem}

The algorithm in Theorem~2 initially chooses some point, $x(0)\in \mathcal{X}$.
It then cycles indefinitely through each index, $i $, and projects onto   $\mathcal{C}_i$.
It can be viewed as an alternating projection algorithm.
The resulting sequence, $\{x(l)\}_{l=0}^\infty$, converges to a point  which is common to all of the sets, $\mathcal{C}_i$ 
and $\overline{\mathcal X}$.
In addition, if $\nabla f(x(0))$ satisfies \eqref{eq:range}, then $x^*$  is a minimizer for optimization problem \eqref{eq:opt0}.
This algorithm is very useful when one can readily solve the optimization problems in \eqref{eq:opti}.\newline

The optimization problem of this paper \eqref{eq:GenOpt} can be rewritten as
$\min_{(T,v) \in \overline{\mathcal{TV}} }J(T,v)$,
subject to
\begin{align}
\label{eq:rowconstraints3a}
\sum_{j=1}^n t_{ij} = \tilde{u}_i, \, i=1,\dots, m,
 \qquad \sum_{i=1}^m t_{ij} = v_j  \, , j=1,\dots n,
\end{align}
where
\begin{align}
\label{eq:J(T,v)}
J(T,v) &=
KL(T|K) + \gamma KL(v|\tilde{v})
= KL(T|K) + KL(\gamma v|\tilde{\gamma v})
\end{align} 
 $\gamma >0$ and
  {\small
\begin{align*}
\label{eq:Script2}
\mathcal{S} =& \left\{ (T,v) : T \in \mathbb{R}^{m\times n}, v \in \mathbb{R}^n  \text{ and }  t_{ij} = 0 
\;\forall \; (i,j) \in \mathcal{Z} \right\}
\\
\mathcal{TV} =& \{(T, v) \in \mathcal{S} :  t_{ij} >0
\,\forall\, (i,j) \notin \mathcal{Z} \text { and } v_j > 0\, \forall\, i \}
 \end{align*}
 }
 Let $q= n_T + n$, where 
 $n_T $ is the number of index pairs $(i,j)$ not in $\mathcal{Z}$.
Then, by   appropriate definition of $x$  in $\mathbb{R}^q$,
one can associate each element of $x$ to an element of $T$ or $\gamma v$.
We denote this by $x= \text{vec}(T, \gamma v)$
 and the objective function in \eqref{eq:J(T,v)}   can be written as
$f(x) =KL(x |\tilde{x})$,
where $\tilde{x} = \text{vec}(K, \gamma\tilde{v})$.
Also, the constraints in \eqref{eq:rowconstraints3a} can be expressed as
$A_1 x = b_1, \;\; A_2 x = b_2$,
where
\begin{align*}
A_1x &= b_1 \quad \text{iff} \quad (T,v) \in \tilde{\mathcal C}_1
\equiv \left\{(T,v)  \in \mathcal{S}: T {\bm 1} = \tilde{u} \right\}\\
A_2x &= b_2 \quad \text{iff} \quad  (T,v) \in 
\tilde{\mathcal C}_2 \equiv \left\{(T,v) \in \mathcal{S} : T' {\bm 1} = v \right\}
\end{align*}

If $(T, v), (S,w) \in \mathcal{TV}$,  $y=\text{vec}(T, \gamma v)$, $x=\text{vec}(S, \gamma w)$:
\[
KL(y|x) 
= KL(T|S) +\gamma KL(v|w)
 \]

\begin{lemma}
\label{lem:lem2}
If $(S, w) \in \mathcal{TV}$ then
\[
(T^*, v^*) = \text{argmin}_{(T, v) \in \mathcal{TV}\cap \tilde{\mathcal C}_1 }  KL(T|S) +\gamma KL(v|w)
\]
iff
$t^*_{ij} = c_{1i} s_{ij} ,
\;\; v^*_j  = w_j$,
where
\begin{equation}
\label{eq:c1i0}
 c_{1i}= \frac{\tilde{u}_i}{\Sigma_{j=1}^ns_{ij}}
\end{equation}
\end{lemma}
\begin{proof}%
Since  $t_{ij}= 0$ for $(i,j) \in \mathcal{Z}$, the Lagrangian associated with this optimization  problem is
\begin{align*}
L(T, v,\alpha)  &=  KL(T|S) +\gamma KL(v|w)
+\sum_{i=1}^m \alpha_i \left(\sum_{j=1}^n t_{ij}- \tilde{u}_i \right)	
\end{align*}
When $(i,j) \notin \mathcal{Z}$,
\begin{align*}
\frac{\partial L}{\partial t_{ij}}  =    \log \left( \frac{t_{ij}}{s_{ij} }  \right)  +\alpha_i \qquad
\frac{\partial L}{\partial v_j}  =  \gamma \log \left( \frac{v_j}{w_j}\right) 
\end{align*}
Setting these to zero yields
\begin{align*}
\log \left( \frac{t^*_{ij}}{s_{ij} }  \right) =  -\alpha_i\qquad
\log \left( \frac{v^*_j}{w_j}\right)  = 0
\end{align*}
Hence,
$
t^*_{ij}= s_{ij}\exp \left( -\alpha_i \right)	
 = c_{1i}s_{ij}	
$,
where
$
 c_{1i} = \exp(   -\alpha_i)
$
and
$
v^*_j = w_j
$.
Also, $\sum_{j=1}^n t^*_{ij} = \tilde{u}_i$ implies \eqref{eq:c1i0}.
\end{proof}
\vspace{.5em}

%
%
\begin{lemma}
\label{lem:lem3}
If $(S, w) \in \mathcal{TV}$ then
\[
(T^*, v^*) = \text{argmin}_{(T, v) \in \mathcal{TV}\cap \tilde{\mathcal C}_2 } KL(T|S) +\gamma KL(v|w)
\]
iff 
$t^*_{ij} =c_{2j}s_{ij} ,	\;\; v^*_j =  c_{2j}^{-\frac{1}{\gamma}} w_j $,
where
\begin{align}
\label{eq:c2j0}
c_{2j} = \left(\frac{w_j}{\Sigma_{i=1}^m s_{ij}}\right)^\frac{\gamma}{1+\gamma }
\end{align}
\end{lemma}
\begin{proof}
Since  $t_{ij}= 0$ for $(i,j) \in \mathcal{Z}$,  the
Lagrangian associated with this optimization  problem is given by
\begin{align*}
L(T, v, \beta)  &=    KL(T|S) +\gamma KL(v|w)
 +\sum_{j=1}^n\beta_j \left(\sum_{i=1}^m t_{ij}- v_j\right)	
\end{align*}
When $(i,j) \notin{\mathcal Z}$,
\begin{align*}
\frac{\partial L}{\partial t_{ij}}  =   \log \left( \frac{t_{ij}}{s_{ij} }  \right) +\beta_j  \qquad
\frac{\partial L}{\partial v_j}  =    \gamma \log \left( \frac{v_j}{w_j}\right) - \beta_j  
\end{align*}
Setting these to zero results in
\begin{align*}
\log \left( \frac{t^*_{ij}}{s_{ij} }  \right) =  -\beta_j\qquad
\quad \log \left( \frac{v^*_j}{w_j}\right)  =  \beta_j /\gamma 
\end{align*}
Hence,
$
t^*_{ij}
 = s_{ij}c_{2j} 	
$
where
$
  c_{2j} =\exp  (- \beta_j)
$.
Also,
$
v^*_j = w_{j}\exp(\beta_j /\gamma)	
=  c_{2j} ^{-1/\gamma}w_{j} 
$.
Due to the column sum  constraints, we  must have
$
c_{2j}\sum_{i=1}^ms_{ij} = \sum_{i=1}^m t^*_{ij} =  v^*_j =c_{2j} ^{-1/\gamma}w_{j} 
$,
that is,
$
c_{2j}^{\frac{1+ \gamma}{\gamma}}\sum_{i=1}^ms_{ij}= w_j
$,
which implies \eqref{eq:c2j0}.

\end{proof}

\subsection{Proof of Theorem \ref{th:MSKA}}
\label{sec:prth1}
It follows, from Theorem \ref{th:Bregman} and Lemmas \ref{lem:lem2} and \ref{lem:lem3}, that the sequence,  $\{(T(l),v(l)\}$, converges to a limit, $(T^*, v^*)$,  with $T^* \in {\mathcal T}$, $v^* \in \mathbb{R}^n_+$,
and this limit satisfies the constraints
in \eqref{eq:rowconstraints3a}. We now only need to prove 
that $\nabla f(x(0))$ satisfies \eqref{eq:range}, where $x(0) = \text{vec}(K, \gamma \tilde{v})$.
With $J$ given by \eqref{eq:J(T,v)}, we have, for $t_{ij} >0$ and $v_j >0$,
\[
\frac{\partial J}{\partial t_{ij}}(T,v) = 
   \log \left( \frac{t_{ij}}{k_{ij} } \right) ,
\qquad \frac{\partial J}{\partial \gamma v_j}(T,v) =\log \left(\frac{\gamma v_j}{\gamma\tilde{v}_j}\right) 
\]
Hence,
$\frac{\partial J}{\partial t_{ij}}(K, \tilde{v}) =  0,\;\; \frac{\partial J}{\partial \gamma v_j}(K, \tilde{v}) =0$,
and
$\nabla f(x(0))_i = 0$ for  $i=1,2, \dots , q$. Therefore, $\nabla f(x(0))$ satisfies \eqref{eq:range}.

$\protect{\hspace*{7.75cm}} \blacksquare $

\section{Example: resource allocation in the sharing economy}
\label{sec:experiment}
SK-type iterations make sense for large-scale problems. Smart cities are a natural place to look for such problems, and, in particular,   sharing economy \cite{shorten} applications, since these are precisely the problem domains where allocation of resources at scale emerge, and where the scale of the problem is subject to temporal variations. One such problem arises in the context of charging $m$ electric vehicles (EVs) overnight in a city such as London. With the advent of vehicle-to-grid (V2G), vehicle-to-vehicle (V2V) and widespread availability of solar, it is likely that many entities that currently consume energy will become prosumers in the near future; i.e.\ most homes, and even cars, will  consume energy and also make energy available, depending on the circumstance. Such a scenario is clearly very large scale, with a potentially very large number, $m$, of EVs (there are currently more than 2.5M cars registered in London), and a large number, $n$, of energy providers (consisting of conventional utilities, energy brokerages, households and even other cars, and constituting the target agents of our OT setup in Section~\ref{sec:OTzero}). In the setting of this paper (Section~\ref{sec:OTzero}),  each EV specifies the energy it requires in KWhrs, and then this demand for energy is communicated (i.e.\ `transported') to a set of providers.  Each provider may set a cost based on their type of energy generation, their proximity to the car being charged, and the type of vehicle being charged. In addition, some providers may prohibit certain types of vehicles, for example plug-in hybrid vehicles (PHEVs), or vehicles that are very large in size. Our OT formulation (\ref{eq:GenOpt}) captures the realistic scenario in which the $n$ energy providers make available  flexible amounts of energy, nominally  $\tilde{v}$, to $m$ EVs whose demands are  exactly  $\tilde{u}$. Furthermore, energy transfer between specified provider-EV pairs, $(i,j) \in {\mathcal Z}$ (\ref{eq:U}), are not  allowed {\em a priori}.\newline

To simulate this scenario, consider $m=10,000$ EVs requiring charging and $n = 10$  energy providers. We simulate  the specified energy charging requirement, $\tilde{u}_i$, $i=1, \ldots, m$, of each EV via independent and uniform draws in the range $(0,1)$ KWhrs. We simulate the nominal available energy of each supplier, $\tilde{v}_j$, $j=1, \ldots, n$, in the same way. Therefore, (\ref{eq:ut=vt}) may not hold, i.e.\ the transport problem may be unbalanced, in that the total  energy required by the EVs may differ from the total nominal energy  made available by the providers. This UOT problem is also solved by our algorithms.  To effect prior zero constraints, we  assume that even-indexed cars are PHEVs, and that even-indexed providers will not supply these PHEVs. This defines a transport plan with $\frac{mn}{4}$ pre-specified zeroes. The non-zero elements of $K$ in (\ref{eq:K}) are obtained with $\tilde{t}_{ij} = 1$ and $\gamma_0 = 1.99$, 
and the transport cost, $c_{ij}$, $
\forall (i,j)\notin \mathcal {Z} $  are again iid  uniformly drawn from  $(0,1)$. We use Algorithm 1, with 
  $\gamma = 1.005$ (\ref{eq:GenOpt}), to obtain the OT plan, $T^*$. 
Figure  \ref{fig:opt} illustrates the convergence of the algorithm for 10 simulation runs. 
 Let $T(l-1)$ and $T(l)$ be the transport plans obtained from two consecutive iterations of the algorithm in any one simulation, and let $\Delta(l)$ be the matrix with $(i,j)$th entry  defined as $\delta_{ij}(l) = | t_{ij}(l)-t_{ij}(l-1) |$.
Figure  \ref{fig:opt}  plots  the log of the sum of the $\delta_{ij}$s, normalised by the first {log-sum}, for each of the 10 simulations,
{ i.e.
$\frac{\log\left(\sum_{(i,j) \notin {\mathcal Z} }|t_{ij}(l) - t_{ij}(l-1)| \right)}
{\log\left( \sum_{(i,j) \notin {\mathcal Z} }  |t_{ij}(1) - t_{ij}(0)| \right) }$.
We have normalized in order to facilitate comparison of the different simulations.
}
\begin{figure}[th]
        \centering
             \includegraphics[width=0.4\textwidth]{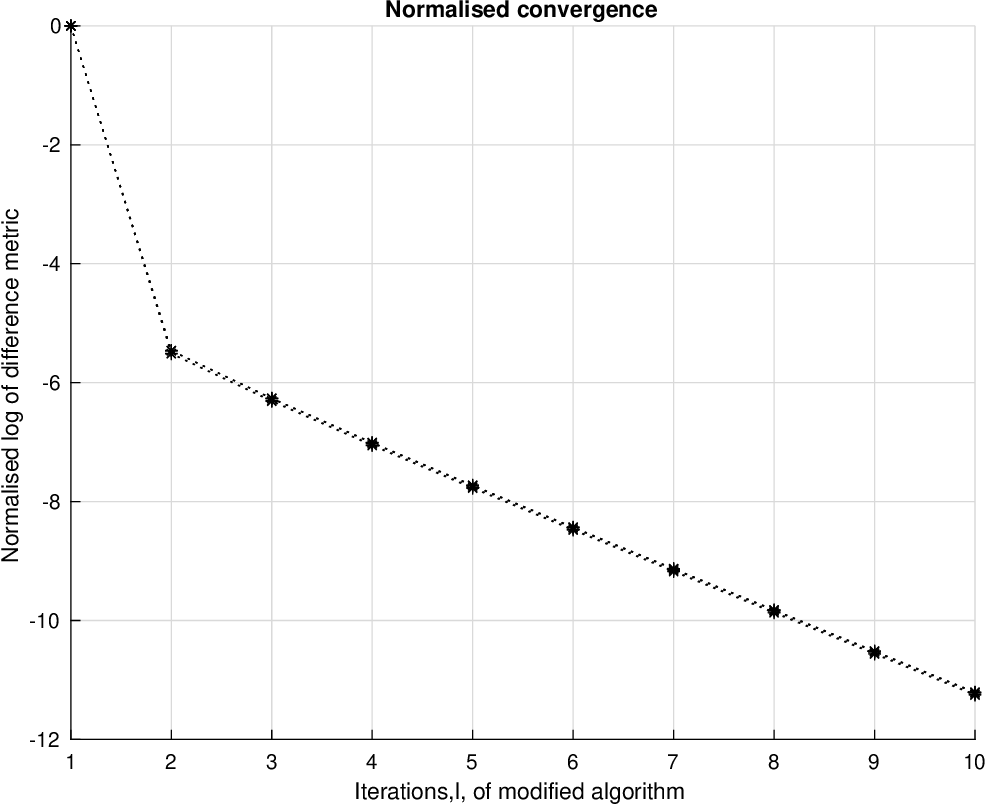}
       \caption{Convergence of the modified algorithm (Algorithm~1) }
        \label{fig:opt}
 \end{figure}



\vspace*{-.9cm}
\section{Conclusions}
\label{sec:conc}
In this paper, we have presented SK-type algorithms for constrained optimal transport. Specifically, our algorithms allow for transport plans that force some entries to be zero {\em a priori}.  The convergence proof  is based on Bregman-type ideas.  An example in resource allocation for the sharing economy is provided,  pointing to situations in which our algorithm is relevant. {Future work will investigate extension of this work to problems that incorporate other forms of regularisation terms \cite{FerradansAl2014}, and other potential use-cases.} 

\bibliographystyle{ieeetr}

\bibliography{bibliography,optTrans}



\end{document}